\documentclass[12pt,reqno]{amsart}

\usepackage[dvips]{graphicx}
\usepackage[arrow,matrix,curve]{xy} 

\usepackage{amssymb, latexsym, amsmath, amscd, array, 
}

\usepackage{color}
\definecolor{darkgreen}{cmyk}{1,0,1,.2}
\definecolor{m}{rgb}{1,0.1,1}

\newtheorem{theorem}{Theorem}[section]

\theoremstyle{definition}

\newtheorem{remark}[theorem]{Remark}

\numberwithin{equation}{section}
\numberwithin{figure}{section} 
\numberwithin{table}{section}

\newcommand\st{{\rm st}} 

\newcommand\SO{{\operatorname{SO}}}

\newcommand\R {{\mathbb R}}

\newcommand\N {{\mathbb N}} 

\newcommand\Q{{\mathbb {Q}}}

\newcommand\RRR{\mathbb{I\hskip-3pt R}}
\newcommand\Los{{\L}o{\'s}}

\newcommand\fnull{\mathcal{F}_{\!n\!u\!l\!l}}
\newcommand\fez{\mathcal{F}_{ez}}

\begin{document}

\title{Commuting and non-commuting infinitesimals}

\author[M.~Katz]{Mikhail G. Katz} \address{Department of Mathematics,
Bar Ilan University, Ramat Gan 52900 Israel}
\email{katzmik@macs.biu.ac.il}

\author[E. Leichtnam]{Eric Leichtnam} \address{CNRS et Institut de
Jussieu-Chevaleret, Etage 7 E, 175 Rue du Chevaleret, 75013 Paris}
\email{ericleichtnam(strudel)math.jussieu.fr}




\begin{abstract}
Infinitesimals are natural products of the human imagination.  Their
history goes back to the Greek antiquity.  Their role in the calculus
and analysis has seen dramatic ups and downs.  They have stimulated
strong opinions and even vitriol.  Edwin Hewitt developed hyperreal
fields in the 1940s.  Abraham Robinson's infinitesimals date from the
1960s.  A noncommutative version of infinitesimals, due to Alain
Connes, has been in use since the 1990s.  We review some of the
hyperreal concepts, and compare them with some of the concepts
underlying noncommutative geometry.
\end{abstract}

\maketitle 


\section{A brief history of infinitesimals}
\label{one}

A theory of infinitesimals was developed by Abraham Robinson in the
1960s (see \cite{Ro61}, \cite{Ro66}).  In France, Robinson's lead was
followed by the analyst G.~Choquet%
\footnote{See e.g., Choquet's work on ultrafilters \cite{Ch}.
Choquet's constructions were employed and extended by
Mokobodzki~\cite{Mo}.}
and his group.  Alain Connes started his work under Choquet's
leadership.  In 1970, Connes published two texts on hyperreals and
ultrapowers \cite{Co70a, Co70b}.  Some time after Robinson's death in
1974, Connes developed an alternative theory of infinitesimals.
Connes' presentation of his theory is frequenly accompanied by
criticism of Robinson's.

Our goal is to compare the concept of an infinitesimal in Robinson's
hyperreals, with an analogous concept in Connes' noncommutative
geometry.  We shall also review several comments by Connes on the
subject of Robinson's hyperreals, such as the following comment made
in~2000:
\begin{quote}
A nonstandard number is some sort of chimera which is impossible to
grasp and certainly not a concrete object (Connes \cite{Co00A},
\cite{Co00B}).
\end{quote}
D.~Tall describes a concept in cognitive theory he calls a {\em
generic limit concept\/} in the following terms:
\begin{quote}
if a quantity repeatedly gets smaller and smaller and smaller without
ever being zero, then the limiting object is naturally conceptualised
as an extremely small quantity that is not zero (Cornu \cite{Co91}).
Infinitesimal quantities are natural products of the human imagination
derived through combining potentially infinite repetition and the
recognition of its repeating properties (Tall \cite{Ta91},
\cite{Ta09a}).
\end{quote}

Leibniz based the fundamental concepts of the calculus on
infinitesimals, an approach that was followed by l'H\^opital, Johann
Bernoulli, Varignon, and others.  Thus, the ``differential quotient''
(which evolved into the modern derivative) was thought of as a ratio
of infinitesimals, while the integral, as an infinite sum of areas of
infinitely narrow rectangles.

In the context of European mathematics, infinitesimals were already
uppermost in the mind of Nicholas of Cusa in the 15th century.
Nicholas of Cusa thought of the circle as a polygon with infinitely
many sides, inspiring Kepler to formulate his ``bridge of continuity"
(see Baron~\cite[p.~110]{Ba}).  The law of continuity was a heuristic
principle developed by Leibniz.  He formulated it as follows in a 1702
letter to Varignon:
\begin{quote}
The rules of the finite succeed in the [realm of the] infinite
[\ldots] and vice versa the rules of the infinite succeed in the
[realm of the] finite (Leibniz 1702, \cite{Le02}; cf.~Robinson
\cite[p.~262, 266]{Ro66}; Knobloch \cite[p.~67]{Kn}).
\end{quote}
Robinson put it as follows in 1966:
\begin{quote}
Leibniz did say [$\ldots$] that what succeeds for the finite numbers
succeeds also for the infinite numbers and vice versa, and this is
remarkably close to our transfer of statements from~$\R$ to~${}^*\R$
and in the opposite direction.  But to what sort of laws was this
principle supposed to apply? \cite[p.~266]{Ro66} (see also Laugwitz
\cite{Lau92}).
\end{quote}


A century after Leibniz, L.~Carnot and A.-L.~ Cauchy still exploited
the concept of an infinitesimal, generated by a suitable variable
quantity, namely a null sequence (sequence tending to zero).  An
alternative foundation for the calculus in terms of the epsilon, delta
approach was developed by G.~Cantor, R.~Dedekind, and K.~Weierstrass
starting in the 1870s.  As an example of the epsilon-delta method,
consider Cauchy's definition of continuity of a function~$y=f(x)$:
\begin{quote}
{\em an infinitesimal change~$\alpha$ of the independent variable~$x$
always produces an infinitesimal change~$f(x+\alpha)-f(x)$ of the
dependent variable~$y$\/} (cf.~Cauchy \cite[p.~34]{Ca21}).%
\footnote{For details on Cauchy's infinitesimals, see Cutland et
al.~\cite{CKKR}, B\l aszczyk et al.~\cite{BKS}, Borovik et
al.~\cite{BK}.}
\end{quote}
Weierstrass reconstructs Cauchy's infinitesimal definition as follows:
for every~$\epsilon>0$ there exists a~$\delta>0$ such that for
every~$\alpha$, if~$|\alpha|<\delta$
then~$|f(x+\alpha)-f(x)|<\epsilon$.

Before infinitesimals would finally be justified in a mathematically
rigorous fashion by Robinson \cite{Ro66}, they were yet to be derided
as ``paper numbers'', ``cholera bacillus'' of mathematics, and an
``abomination'' by Georg Cantor (see Meschkowski \cite[p.~505]{Mes}),
who had convinced himself of the impossibility of justifying
infinitesimals in such a fashion (see J.~Dauben~\cite{Da96},
P.~Ehrlich \cite{Eh06}, and B\l aszczyk et al.~\cite{BKS} for
details).

\section{Robinson's framework}
\label{rob}

In 1961, Robinson \cite{Ro61} constructed an infinitesimal-enriched
continuum, suitable for use in calculus, analysis, and elsewhere,
based on earlier work by E.~Hewitt \cite{Hew}, J.~\Los{} \cite{Lo},
and others.  In 1962, W.~Luxemburg \cite{Lu62} popularized a
presentation of Robinson's theory in terms of the ultrapower
construction,%
\footnote{Note that both the term ``hyper-real'', and an ultrapower
construction of a hyperreal field, are due to E.~Hewitt in 1948, see
\cite[p.~74]{Hew}.  Luxemburg~\cite{Lu62} also clarified its relation
to the competing construction of Schmieden and Laugwitz \cite{SL},
similarly based on sequences, which used a different kind of filter.}
in the mainstream foundational framework of the Zermelo--Fraenkel set
theory with the axiom of choice (ZFC).  Namely, the hyperreal field is
the quotient of the collection of arbitrary sequences, where a
sequence
\begin{equation}
\label{11}
\langle u_1, u_2, u_3, \ldots \rangle
\end{equation}
converging to zero generates an infinitesimal.  Arithmetic operations
are defined at the level of representing sequences; e.g., addition and
multiplication are defined term-by-term.

To motivate the construction of the hyperreals, note that the
construction can be viewed as a relaxing, or refining, of Cantor's
construction of the reals.  This can be motivated by a discussion of
rates of convergence as follows.  In Cantor's construction, a real
number $u$ is represented by a Cauchy sequence $\langle u_n :
n\in\N\rangle$ of rationals.  But the passage from $\langle
u_n\rangle$ to $u$ in Cantor's construction sacrifices too much
information.  We would like to retain a bit of the information about
the sequence, such as its ``speed of convergence".  This is what one
means by ``relaxing" or ``refining" Cantor's construction
(cf.~Giordano et al.~\cite{GK11}).  When such an additional piece of
information is retained, two different sequences, say $\langle
u_n\rangle$ and $\langle u'_n\rangle$, may both converge to $u$, but
at different speeds.  The corresponding ``numbers" will differ from
$u$ by distinct infinitesimals.  If $\langle u_n\rangle$ converges to
$u$ faster than $\langle u'_n\rangle$, then the corresponding
infinitesimal will be smaller.  The retaining of such additional
information allows one to distinguish between the equivalence class of
$\langle u_n\rangle$ and that of $\langle u'_n\rangle$ and therefore
obtain different hyperreals infinitely close to $u$.

At the formal level, we proceed as follows.  We start with the
ring~$\Q^\N$ of sequences of rational numbers.  Let
\begin{equation}
\label{22}
\mathcal{C}_\Q \subset \Q^\N
\end{equation}
denote the subring consisting of Cauchy sequences.  The reals are by
definition the quotient field
\begin{equation}
\label{realbis}
\R:= \mathcal{C}_\Q / \fnull,
\end{equation}
where~$\fnull$ contains all null sequences (i.e., sequences tending to
zero).  An infinitesimal-enriched extension of~$\Q$ may be obtained by
modifying~\eqref{realbis} as follows.  We consider a
subring~$\fez\subset\fnull$ of sequences that are ``eventually zero'',
i.e., vanish at all but finitely many places.  Then the
quotient~$\mathcal{C}_\Q / \fez$ naturally surjects onto~$\R=
\mathcal{C}_\Q / \fnull$.  The elements in the kernel of the
surjection
\[
\mathcal{C}_\Q/\fez \to \R
\]
are prototypes of infinitesimals.  Note that the
quotient~$\mathcal{C}_\Q/\fez$ is not a field, as~$\fez$ is not a
maximal ideal.  The natural next step is to replace~$\fez$ by a
\emph{maximal} ideal.  It is more convenient to describe the modified
construction using the ring~$\R^\N$ rather than $\mathcal{C}_\Q$
of~\eqref{22}.

We therefore redefine~$\fez$ to be the ring of real sequences
in~$\R^\N$ that eventually vanish, and choose a \emph{maximal} proper
ideal $\mathcal{M}$ so that we have
\begin{equation}
\label{23}
\fez\subset\mathcal{M}\subset\R^\N.
\end{equation}
Then the quotient~$\RRR:=\R^\N/\mathcal{M}$ is a hyperreal field.  The
foundational material needed to ensure the existence of a maximal
ideal~$\mathcal{M}$ satisfing~\eqref{23} is weaker than the axiom of
choice.  This concludes the construction of a hyperreal field~$\RRR$
(``thick-$\R$'') in the traditional foundational framework, ZFC.

Let~$I\subset\RRR$ be the subring consisting of infinitesimal elements
(i.e., elements~$e$ such that~$|e|<\frac{1}{n}$ for all~$n\in\N$).
Denote by~$I^{-1}$ the set of inverses of nonzero elements of~$I$.
The complement~$\RRR\setminus I^{-1}$ consists of all the finite
(sometimes called \emph{limited}) hyperreals.  Constant sequences
provide an inclusion~$\R\subset\RRR$.  Every element of the
complement,~$x\in \RRR\setminus I^{-1}$ is infinitely close to some
real number $x_0\in\R$.  The \emph{standard part function}, denoted
``st'', associates to every finite hyperreal, the unique real
infinitely close to it:
\[
\text{st}:\RRR\setminus I^{-1} \to \R, \text{\;with\;\;} x \mapsto
x_0.
\]
If $x$ happens to be represented by a Cauchy sequence $\langle
x_n:n\in\N\rangle$ then $x_0=\st(x)=\lim_{n\to\infty}x_n$.  More
advanced properties of the hyperreals such as saturation were proved
later (see Keisler \cite{Ke94} for a historical outline).  A helpful
``semicolon'' notation for presenting an extended decimal expansion of
a hyperreal was described by A.~H.~Lightstone~\cite{Li}.  See also
P.~Roquette \cite{Roq} for infinitesimal reminiscences.  A discussion
of infinitesimal optics is in K.~Stroyan \cite{Str},
J.~Keisler~\cite{Ke}, D.~Tall~\cite{Ta80}, and L.~Magnani and
R.~Dossena~\cite{MD, DM}.  Applications of infinitesimal-enriched
continua range from aid in teaching calculus \cite{El, KK1, KK2} to
the Bolzmann equation (see L.~Arkeryd~\cite{Ar81, Ar05}), mathematical
physics (see Albeverio {\em et al.\/} \cite{Alb}); etc.  Edward Nelson
\cite{Ne} in 1977 proposed an alternative to ZFC which is a richer
(more stratified) axiomatisation for set theory, called Internal Set
Theory (IST), more congenial to infinitesimals than ZFC.  A rethinking
of the history, mathematics, and philosophy of infinitesimals has been
undertaken in \cite{BKS}, \cite{BJK}, \cite{BK}, \cite{KK11a},
\cite{KK11b}, \cite{KK11c}, \cite{KK11d}, \cite{KS}, \cite{KT}.
Recently, P.~Ehrlich \cite[Theorem~20]{Eh12} showed that the ordered
field underlying a maximal (i.e., \emph{On}-saturated) hyperreal field
is isomorphic to J. H. Conway's ordered field No, an ordered field
Ehrlich describes as the \emph{absolute arithmetic continuum}.

\section{From physics to noncommuting infinitesimals}
\label{two}

One of the revolutionary observations of 20th century physics is that
observables cannot take scalar values.  Experiments have shown that
position~$X$ and momentum~$P$ satisfy the relation (frequently
described as the uncertainty principle)
\[
[X,P]=i \hbar ,
\]
which expresses such an absence of commutation.

Since addition and multiplication of sequences is term-by-term, the
multiplication discussed in Section~\ref{one} is still commutative.
Now think of the sequence~\eqref{11} as a diagonal matrix
\[
{\rm diag}(u_1, u_2, u_3, \ldots)
\]
of infinite size.  We now enlarge the collection of matrices, by
propagating by an action of~$\SO(\infty)$.  We obtain a vast
collection of noncommuting matrices.  The original sequence can be
retrieved by symmetrisation, followed by applying Spec (the spectrum,
i.e.~calculating the eigenvalues).

At this point, a mathematician worth his ilk supresses the dependence
on the preferred basis, spruces up \emph{matrices} as
\emph{operators}, and recalls a basic fact from functional analysis:
an operator whose~$n$-th eigenvalue tends to zero, is a compact
operator (image of the unit ball is compact).

In practice, Connes chooses a much smaller collection of compact
operators as his class of infinitesimals.  It is defined by sequences
with a specific rate of convergence to~$0$.

Connes published two papers \cite{Co70a, Co70b} on non-standard
analysis in 1970.  The paper \cite{Co70a} is specifically devoted to
the ultrapower approach.  The connection between Connes'
infinitesimals and Robinson's via the sequence defined by the spectrum
is an intriguing one, see Albeverio et al.~\cite{Al96}.%
\footnote{Yamashita \cite{YO, Ya} applied Robinson's hyperreals to
quantum field theory, but did not emphasize the relation to Connes'
noncommutative geometry.}

\section{Noncommutative geometry and non-standard analysis}

Noncommutative geometry is one of the fastest growing theories today,
influential both in mathematics and in physics.  One of the technical
aspects of noncommutative geometry, which plays a role, in particular,
in establishing a mathematical framework for high-energy physics, is
the Dixmier trace.

Both Connes' and Robinson's theories involve a concept of an
infinitesimal, leading to a natural question as to the comparison of
the two.

Note that a compact operator is by definition one whose~$n$-th
eigenvalue tends to zero (at least modulo symmetrisation).  This
motivates viewing a compact operator as an infinitesimal (see
Section~\ref{two}).

Connes exploited the Dixmier trace \cite{Di66} to give a uniform
explanation of the pseudodifferential residue%
\footnote{No understanding of the pseudodifferential residue is
required for understanding the present text.}
of Guillemin and Wodzicki.  The result was immediately hailed as a
major accomplishment.  The exotic traces constructed by J.~Dixmier are
not normal, and have the property of being nonnegative on a compact
positive operator.  In his article {\em Brisure de sym\'etrie\/},
Connes argues that his solution is \emph{substantial and
calculable}~\cite[p.~211]{Co97}, and expresses a disappointment with
an allegedly non-exhibitable nature of Robinson's infinitesimals.
Meanwhile, Dixmier's construction of the trace relies on the choice of
an ultrafilter on the integers.%
\footnote{Connes has shown that a theorem of Mokobodzki
(see~\cite{Me}) provides a limit process (for the Dixmier trace) which
is universally measurable,
%
%
while relying on the continuum hypothesis (no understanding of either
universal measurability or the continuum hypothesis is required in the
sequel).  See Remark~\ref{CH} for more details.}

It should be stated at the outset that the focus of interest of
Connes' noncommutative geometry lies elsewhere.  Still, one of the
building blocks of Connes' theory is a framework incorporating
non-commuting infinitesimals.  In his articles published in refereed
journals, Connes repeatedly stresses the role of Dixmier's traces in
implementing a general framework for working with his infinitesimals.
It is significant to ponder the fact itself of the existence of
Connes' implementation, based, as it is, on Dixmier traces, whose
foundational status is dependent on a choice of an ultrafilter,
closely related to the axiom of choice.

While it has been pointed out that in concrete instances, the Dixmier
trace is replaced by more concrete objects whose foundational status
is more constructive (such a claim ought to be taken with a grain of
salt, as developments in functional analysis frequently involve an
unspoken dependence on such results as the Hahn-Banach theorem,%
\footnote{See Remark~\ref{CH} illustrating Connes' reliance on the
Hahn-Banach theorem.}
similarly non-constructive; be that as it may), it is significant that
Connes indisputably succeeds in implementing a general framework based
on Dixmier traces, inspite of their non-constructive foundational
status.

In this context, it is instructive to examine Connes' published
comments on Robinson's theory, which typically go hand in hand with
Connes' acknowledgment of indebtedness to Dixmier and his traces.%
\footnote{Note that Robinson studied generalized limits in 1964, in
the context of (infinite) real Toeplitz matrices (see~\cite{Ro64});
however a key property of {\em scale invariance\/} is apparently not
present.}
The importance of Dixmier traces in noncommutative geometry was noted
by S.~Albeverio et al.~\cite{Al96}.

\section{Connes' three objections to Robinson's theory}

Connes has expressed himself on a number of occasions on the subject
of Robinson's theory.  Thus, in a 2000 interview, he said:
\begin{quote}
I became aware of an absolutely major flaw in this theory, an
irremediable defect (Connes \cite[p.~16]{CLS}).
\end{quote}
In 2007, Connes worded himself in a more nuanced way:
\begin{quote}
it seemed utterly doomed to failure to try to use non-standard
analysis to do physics (Connes \cite[p.~26]{Co07}),
\end{quote}
apparently implying that the alleged shortcomings of Robinson's theory
are limited specifically to potential applications in \emph{physics}.%
\footnote{Recently, S. Kutateladze reacted as follows to Connes'
comment: ``\emph{Physics} meant \emph{mechanics} for about two
centuries.  Newton and Leibniz and their followers used heuristic
infinitesimals, which were ultimately implemented mathematically by
Robinson. Celestial mechanics and hydrodynamics still require
infinitesimals in much the same way as in Leibniz."}

In 1995, Connes gave a detailed account of the role of the Dixmier
trace in his theory.  Connes states that the goal \cite[section
II]{Co95} is to develop a ``calculus of infinitesimals'' based on
operators in Hilbert space (see Section~\ref{two} above), and proceeds
to
\begin{quote}
explain why the formalism of nonstandard analysis is inadequate
(Connes \cite[p.~6207]{Co95}).
\end{quote}
Connes points out the following three aspects of Robinson's hyperreals
(the explicit list is ours, not Connes'):
\begin{enumerate}
\item
a nonstandard hyperreal ``cannot be exhibited" (the reason given being
its relation to non-measurable sets);
\item
``the practical use of such a notion is limited to computations in
which the final result is independent of the exact value of the above
infinitesimal.  This is the way nonstandard analysis and ultraproducts
are used";
\item
the hyperreals are commutative.
\end{enumerate}

We will argue that two of the three arguments given by Connes with
regard to the inadequacy of Robinson's theory may be weaker than he
claims since they appear to apply similarly to his own calculus.  If
so, Connes' claim that his ``theory of infinitesimal variables is
completely different'' may be exaggerated.  For a sequel paper, see
(Katz, Kanovei, Mormann 2013 \cite{KKM}).

Connes proceeds to establish a dictionary on page 6208, relating
classical and quantum notions.  The last quantum item in his
dictionary is the Dixmier trace, corresponding to ``integrals of
infinitesimals of order 1".  On pages 6210-6211, Connes presents a
pair of technical difficulties with the theory, and states:
\begin{quote}
Both of these problems are resolved by the Dixmier trace
(Connes \cite[p.~6211]{Co95}).
\end{quote}
On page 6212, the Dixmier trace~${\rm Tr}_\omega$ is defined as any
limit point of suitable functionals, where ``the choice of the limit
point is encoded by the index~$\omega$".  Connes goes on to state that
\begin{quote}
for measurable operators~$T$, the value of~${\rm Tr}_\omega(T)$ is
\emph{independent} of~$\omega$ and this common value is the
appropriate integral of~$T$ in the new calculus (Connes
\cite[p.~6213]{Co95}) [emphasis added--the authors]
\end{quote}
Such \emph{independence} would seem to relativize the impact of the
objection~(2) raised above, affirming precisely a similar
independence, for~${\rm Tr_\omega}$.

An alternative construction of the trace by Connes, while not explicit
as it relies on the axiom of choice, is satisfactory since it is given
by a limit process which is universally measurable (see \cite{Me});
however the construction depends on the continuum hypothesis (see
Remark~\ref{CH}).  The foundational status of the trace tends to
relativize the objection~(1).

Note that in Connes' approach, an infinitesimal is given by a compact
operator, and of course many of them can be naturally exhibited (this
is at variance with the general framework, since in order to define,
via the Dixmier trace, an interesting trace of the space of
infinitesimals of degree~$1$, one exploits the axiom of choice).
Similarly, many Robinson infinitesimals can also be naturally
exhibited (see Section~\ref{four}).  Thus, two-thirds of Connes'
critique of Robinson's infinitesimal approach can be said to be
incoherent, in the specific sense of not being coherent with what
Connes writes (approvingly) about his own infinitesimal approach.  The
remaining objection (3), namely the non-commutativity of the
hyperreals, is by far the most convincing of the three objections.  As
Connes wrote, ``The uniqueness of the separable infinite dimensional
Hilbert space cures that problem, and variables with continuous range
coexist happily with variables with countable range, such as the
infinitesimal ones.  The only new fact is that they do not commute''
(Connes \cite[Section~2]{Co08}).

Note that the fact of working with non-commutative algebras allows
Connes to deal with many interesting applications: foliations, set of
irreducible representations of a noncommutative discrete group, etc.

In Connes' approach, noncommutativity ensures the coexistence of
variables having a Lebesgue spectrum with infinitesimal variables.  To
elaborate, consider the map~$ x\rightarrow f(x)$ from~$[0,1]$ into
itself, where~$f$ is monotone increasing.  By multiplication it
defines a self adjoint operator from~$L^2([0,1])$ into itself with
Lebesgue spectrum.  Now consider a map~$x \rightarrow A(x)$ from
$[0,1]$ into itself having only a countable set of values, for
instance, the set~$\left\{ \frac{1}{n} :\; n \in
\mathbb{N}\setminus\{0\} \right\}$. Then, for infinitely many values
of~$n$,~$A^{-1}(\{1/n\})$ is uncountable with positive measure.  For
such~$n$,~$1/n$ is an eigenvalue of the multiplication operator by~$A$
whose eigenspace is infinite dimensional.  Thus,~$A$ cannot define an
infinitesimal.  In the noncommutative picture, a variable is given by
a self adjoint operator of~$B L^2([0,1])$, and an infinitesimal is
given by a compact self adjoint operator~$K$ of~$B L^2([0,1])$.  The
set of eigenvalues of~$K$ forms a sequence, it tends to zero and the
eigenspaces associated to nonzero eigenvalues are finite dimensional.

The remarkable coincidence of dates: both Robinson's book \cite{Ro66}
and Dixmier's article~\cite{Di66} were published in~1966, suggests
that a certain cross-fertilisation of ideas may have taken place.  Can
this particular subparagraph of the early version of Noncommutative
geometry, as it appeared in the nineties, be thought of having an
important precursor in Robinson's monumental work \cite{Ro66}, in
addition to Dixmier's 2-page article~\cite{Di66}?

\begin{remark}
\label{CH}
On pages 303-308 of his book \cite{Co94}, Connes presents a detailed
construction of the Dixmier trace.  On page 305, he chooses a positive
linear form~$L$ on the vector space of bounded continuous functions
on~$\R_+\setminus\{0\}$, such that~$L(1)=1$, and which is zero on the
subspace of functions vanishing at infinity.  The construction of the
trace is eventually proved to be independent of such a choice.  The
choice of~$L$ relies on the Hahn-Banach theorem, of similar
foundational status.

On the other hand, using a theorem of Mokobodzki for the Cantor
set~$\{0,1\}^\N$, Connes constructed a Dixmier trace by a limit
process which is universally measurable, based on foundational
material including the continuum hypothesis (CH) (though it is
impossible to exhibit a representative of this Dixmier trace).
However, CH is generally considered to be a strong foundational
hypothesis.  Thus, all ultrapower-type models of the hyperreals
starting from $\R^\N$ become isomorphic if one assumes CH (see
\cite{Eb}).
\end{remark}

\section{Exhibitable Robinson infinitesimal}
\label{four}

On the connection between infinitesimals and non-measurable sets,
Connes writes as follows:
\begin{quote}
Every non-standard real determines canonically a Lebesgue
non-measurable subset of the interval~$[0,1]$, so that it is
impossible to {\bf exhibit} a single one (Connes \cite[p.~211]{Co97})
[emphasis added--authors].
\end{quote}
The expression ``single one'' apparently refers to ``non-standard
real'', and not ``non-measurable set'', and has been widely
interpreted as such in the literature.

Thus, as recently as 2009, M. de Glas claims as a matter of fact that
each hyperreal is canonically associated with a non-measurable subset
of the real line \cite[p.~194]{De}.

Meanwhile, Keisler's {\em Elementary calculus} \cite{Ke} on page~913,
line~3, {\bf exhibits} an explicit representative of an equivalence
class defining an infinitesimal hyperreal (in the context of the
ultrapower construction following Luxemburg).  Are they exhibitable or
non-exhibitable?  The explanation of the apparent paradox is as
follows.

Given an infinite hypernatural~$H$, we can consider all
subsets~$A\subset\N$ whose natural hyperreal
extension~$^{*}\!A\subset{}^{*}\N$ contains~$H$.  The resulting
collection defines an ultrafilter on $\N$ which, viewed as a subset of
[0,1] via dyadic representation, produces ``canonically'' a
non-measurable set.  Similarly, given an arbitrary infinite
hyperreal~$H$, we can consider its integer part%
\footnote{Here~$^{*}[H]$ is the image of $H$ under~$^{*}\![\;\,]$, the
natural extension of the integer part function $[\;\,]$ on $\R$.}
\begin{equation}
\label{61}
^{*}[H],
\end{equation}
and proceed ``canonically'' as before.  Given an
infinitesimal~$\epsilon$, we can consider the
hyperreal
\begin{equation}
\label{51}
H=\frac{1}{\epsilon}
\end{equation} 
and proceed ``canonically'' as in \eqref{61}.  Finally, given a
non-standard finite hyperreal~$x$, we can consider the infinitesimal
difference
\begin{equation}
\label{31}
\epsilon=x-{\rm st}(x),
\end{equation}
where ``st'' is the standard part function (see Section~\ref{rob}),
and proceed ``canonically'' as in~\eqref{51}.

The catch (hitch?) is implicit in the meaning of the word
``canonical''.  The construction is only canonical once one has in
place the powerful new principles of reasoning such as the transfer
principle (a mathematical implementation of Leibniz's heuristic law of
continuity), the standard part function, etc., so that it is possible
to talk about natural extensions of standard objects.

Thus, while it is possible to exhibit a representative of a Robinson
infinitesimal, the choices involved in constructing the hyperreals
make the corresponding ``canonical'' non-measurable set, in fact
nonexhibitable.

In the presence of an ultrafilter construction of the hyperreal line
(see Section~\ref{rob} above as well as (Keisler \cite[p.~911]{Ke}),
one can prove the following result.

\begin{theorem}
A choice of a Connes infinitesimal canonically produces a
non-measurable set.  Such a set cannot be exhibited \cite{St}.
\end{theorem}

\begin{proof}
A Connes infinitesimal is a compact operator~$T$.  The Dixmier trace
is applied to the spectrum of~$|T|$.  The spectrum can be canonically
ordered in decreasing order of~$|\lambda_i|$, producing a
sequence~$(\lambda_i)$ which tends to zero.  As such, it represents a
Robinson infinitesimal \cite[p.~911]{Ke}, and we proceed as in
\eqref{51}.
\end{proof}

\section*{Acknowledgment}

We are grateful to V.~Kanovei, S.~Shelah, and D.~Sherry for helpful
comments.

\bigskip

\textbf{Mikhail G. Katz} is Professor of Mathematics at Bar Ilan
University, Ramat Gan, Israel.  His joint study with P.~B\l aszczyk
and D.~Sherry, \emph{Ten misconceptions from the history of analysis
and their debunking} is due to appear in \emph{Foundations of
Science}.  A joint study with A.~Borovik, \emph{Who gave you the
Cauchy--Weierstras tale?  The dual history of rigorous calculus}
appeared in \emph{Foundations of Science}.  A joint study with
A.~Borovik and R.~Jin, \emph{An integer construction of
infinitesimals: Toward a theory of Eudoxus hyperreals} appeared in
\emph{Notre Dame Journal of Formal Logic} \textbf{53} (2012).


\bigskip

\textbf{Eric Leichtnam} is Director of Research at CNRS at the
\emph{Institut Math\'ematique de Jussieu--Chevaleret} in Paris.  He
has worked in Noncommutative Geometry since 1993.  In 1990 he defended
an \emph{Habilitation \`a Diriger des Recherches} in the area of PDE
at the University of Paris 11.  From 1980 to 1985 he was a student at
the \emph{Ecole Normale Sup\'erieure}.

\end{document}